\documentclass[11pt]{article}
\usepackage{amsfonts,amsmath,amssymb,dsfont,amsthm}
\usepackage[utf8]{inputenc}
\usepackage[T1]{fontenc}
\usepackage{hyperref}
\usepackage{graphicx}
\usepackage{xfrac}
\usepackage{enumitem}

\newcommand{\R}{\mathbb{R}}

\newcommand{\N}{\mathbb{N}}
\newcommand{\M}{\mathcal{M}}
\renewcommand{\P}{\mathcal P}
\newcommand{\A}{\mathcal{A}}

\newcommand{\dis}{\displaystyle}
\newcommand{\p}{\partial}
\newcommand{\f}{\frac}
\newcommand{\beq}{\begin{equation}}
\newcommand{\eeq}{\end{equation}}
\renewcommand{\leq}{\leqslant}
\renewcommand{\geq}{\geqslant}

\newcommand{\1}{{\mathchoice {\rm 1\mskip-4mu l} {\rm 1\mskip-4mu l}
{\rm 1\mskip-4.5mu l} {\rm 1\mskip-5mu l}}}

\def\d {{\mathrm d}}
\def\e{\mathrm{e}}

\DeclareMathOperator{\Ker}{Ker}

\DeclareMathOperator{\supp}{supp}

\newtheorem{theorem}{Theorem}
\newtheorem*{theorem*}{Theorem}
\newtheorem{lemma}[theorem]{Lemma}
\newtheorem{definition}[theorem]{Definition}
\newtheorem*{definition*}{Definition}
\newtheorem{proposition}[theorem]{Proposition}
\newtheorem{corollary}[theorem]{Corollary}
\newtheorem*{remark}{Remark}

\usepackage{colortbl}
\title{Measure solutions to the conservative renewal equation}
\author{Pierre Gabriel \thanks{Laboratoire de Math\'ematiques de Versailles, UVSQ, CNRS, Universit\'e Paris-Saclay,  45 Avenue des \'Etats-Unis, 78035 Versailles cedex, France; e-mail: {\tt pierre.gabriel@uvsq.fr}}}
\date{}

\begin{document}

\maketitle

\begin{abstract}
We prove the existence and uniqueness of measure solutions to the conservative renewal equation
and analyze their long time behavior.
The solutions are built by using a duality approach.
This construction is well suited to apply the Doeblin's argument which ensures the exponential relaxation of the solutions to the equilibrium.
\end{abstract}

\section*{Introduction}

We are interested in the conservative renewal equation
\beq\label{eq:renewal}\left\{\begin{array}{l}
\dis\f{\p n}{\p t}(t,a)+\f{\p n}{\p a}(t,a)+\beta(a)n(t,a)=0,\qquad t,a>0,
\vspace{2mm}\\
\dis n(t,0)=\int_0^\infty\beta(a)n(t,a)\,\d a,\qquad t>0,
\vspace{2mm}\\
n(0,a)=n^{\rm in}(a),\qquad a\geq0.
\end{array}\right.\eeq
It is a standard model of population dynamics, sometimes referred to as the McKendrick-von Foerster model.
The population is structured by an age variable $a\geq0$ which grows at the same speed as time and is reset to zero according to the rate $\beta(a).$
It is used for instance as a model of cell division, the age of the cells being the time elapsed since the mitosis of their mother.
Suppose we follow one cell in a cell line over time and whenever a division occurs we continue to follow only one of the two daughter cells.
Equation~\eqref{eq:renewal} prescribes the time evolution of the probability distribution $n(t,a)$ of the cell to be at age $a$ at time $t,$
starting with an initial probability distribution $n^{\rm in}.$
Integrating (formally) the equation with respect to age we get the conservation property
\[\frac{\d}{\d t}\int_0^\infty n(t,a)\,\d a=0\]
which ensures that, if $n^{\rm in}$ is a probability distribution ({\it i.e.} $\int_0^\infty n^{\rm in}=1$), then $n(t,\cdot)$ is a probability distribution for any time~$t\geq0.$
It is also worth noticing that Equation~\eqref{eq:renewal} admits stationary solutions which are explicitly given by
\[N(a)=N(0)\,\e^{-\int_0^a\beta(u)\d u}.\]
The problem of asymptotic behavior for Equation~\eqref{eq:renewal} consists in investigating the convergence of any solution to a stationary one when time goes to infinity.

\medskip

Age-structured models have been extendedly studied (existence of solutions and asymptotic behavior) for a long time by many authors in a $L^1$ setting
(see for instance among many others~\cite{Iannelli,Thieme,MetzDiekmann,Perthame,Webb}).
More recently measure solutions to structured population models started to draw attention
~\cite{GwiazdaWiedemann,CanizoCarrilloCuadrado,CarrilloColomboGwiazdaUlikowska,EversHilleMuntean15,GwiazdaLorenzMarciniak,DiekmannGetto}.

\medskip

The goal of the mini-course is to define measure solutions to Equation~\eqref{eq:renewal}, prove existence and uniqueness of such solutions, and demonstrate their exponential convergence to the equilibrium.
Considering measure solutions instead of $L^1$ solutions ({\it i.e.} probability density functions) presents the crucial advantage to authorize Dirac masses as initial data.
This is very important for the biological problem since it corresponds to the case when the age of the cell at the initial time is known with accuracy.

\section{Some recalls about measure theory}

We first recall some classical results about measure theory and more particularly about its functional point of view.
For more details and for proofs of the results we refer to~\cite{Rudin}.

We endow $\R_+=[0,+\infty)$ with its standard topology and the associated Borel $\sigma$-algebra.
We denote by $\M(\R_+)$ the set of finite signed Borel measures on $\R_+.$

\begin{theorem*}[Jordan decomposition]
Any $\mu\in\M(\R_+)$ admits a unique decomposition of the form $\mu=\mu_+-\mu_-$
where $\mu_+$ and $\mu_-$ are finite positive Borel measures which are mutually singular.
The positive measure $|\mu|=\mu_++\mu_-$ is called the total variation measure of the measure $\mu.$
We call total variation of $\mu$ the (finite) quantity
\[\|\mu\|_{TV}:=|\mu|(\R_+)=\mu_+(\R_+)+\mu_-(\R_+).\]
\end{theorem*}

We denote by $C_b(\R_+)$ the vector space of bounded continuous functions on $\R_+.$
Endowed with the norm $\|f\|_\infty:=\sup_{x\geq0}|f(x)|$ it is a Banach space.
We also consider the closed subspace $C_0(\R_+)$ of continuous functions which tend to zero at infinity.
To any $\mu\in\M(\R_+)$ we can associate a continuous linear form $T_\mu\in C_b(\R_+)'$ defined by
\[T_\mu:f\mapsto\int_{\R_+}f\,\d\mu.\]
The continuity is ensured by the inequality
\[|T_\mu\, f|\leq\|\mu\|_{TV}\|f\|_\infty.\]
The following theorem ensures that the application $\mu\mapsto T_\mu$ is an isometry from $\M(\R_+)$ onto $C_0(\R_+)',$ where $C_0(\R_+)'$ is endowed with the dual norm
\[\|T\|_{C_0(\R_+)'}:=\sup_{\|f\|_\infty\leq1}|Tf|.\]

\begin{theorem*}[Riesz representation]
For any $T\in C_0(\R_+)'$ there exists a unique $\mu\in\M(\R_+)$ such that $T=T_\mu.$
Additionally we have $\|T\|_{C_0(\R_+)'}=\|\mu\|_{TV}.$
\end{theorem*}

This theorem ensures that $(\M(\R_+),\|\cdot\|_{TV})$ is a Banach space.
It also ensures the existence of an isometric inclusion $C_0(\R_+)'\subset C_b(\R_+)'.$
Notice that this inclusion is strict: there exist nontrivial continuous linear forms on $C_b(\R_+)$ which are trivial on $C_0(\R_+).$
Such a continuous linear form can be built for instance by using the Hahn-Banach theorem to extend the application which associates,
to continuous functions which have a finite limit at infinity, the value of this limit. 

More precisely the mapping $\Psi:C_b(\R_+)'\to C_0(\R_+)'$ defined by $\Psi(T)=T$ is surjective (due to the Riesz representation theorem).
So $C_0(\R_+)'$ is isomorphic to $\sfrac{C_b(\R_+)'}{\Ker\Psi},$
{\it i.e.} for all $T\in C_b(\R_+)'$ there exists a unique decomposition $T=\mu+L$ with $\mu\in C_0(\R_+)'=\M(\R_+)$ and $L\in\Ker\Psi.$
Additionally $T\in C_0(\R_+)'$ if and only if for all $f\in C_b(\R_+)$ we have $Tf=\lim_{n\to\infty}T(f_n)$
where $f_n$ is defined for $n\in\N$ by
\beq\label{eq:f_n}f_n(x)=\left\{\begin{array}{ll}
f(x)&\text{if}\quad 0\leq x\leq n,\\
(n+1-x)f(x)&\text{if}\quad n<x<n+1,\\
0&\text{if}\quad x\geq n+1.
\end{array}\right.\eeq
Abusing notations, we will now denote for $f\in C_b(\R_+)$ and $\mu\in\M(\R_+)$
\[\mu f:= T_\mu\, f=\int_{\R_+} f\,\d\mu.\]
We end by recalling two notions of convergence in $\M(\R_+)$ which are weaker than the convergence in norm.
\begin{definition*}[Weak convergence]
A sequence $(\mu^n)_{n\in\N}\subset\M(\R_+)$ converges narrowly (resp. weak*) to $\mu\in\M(\R_+)$ as $n\to\infty$ if
\[\lim_{n\to\infty}\mu^nf=\mu f\]
for all $f\in C_b(\R_+)$ (resp. for all $f\in C_0(\R_+)$).
\end{definition*}

\section{Definition of a measure solution}

Before giving the definition of a measure solution to Equation~\eqref{eq:renewal},
we need to state the assumptions on the division rate $\beta.$
We assume that $\beta$ is a 
continuous function on $\R_+$ which satisfies
\[\exists\, a_*,\beta_{\min},\beta_{\max}>0,\ \forall a\geq0,\qquad \beta_{\min}\,\mathds{1}_{[a_*,\infty)}(a)\leq\beta(a)\leq\beta_{\max}.\]

We explain now how we extend the classical sense of Equation~\eqref{eq:renewal} to measures.
Assume that $n(t,a)\in C^1(\R_+\times\R_+)\cap C(\R_+;L^1(\R_+))$ satisfies~\eqref{eq:renewal} in the classical sense,
and let $f\in C^1_c(\R_+)$ the space of continuously differentiable functions with compact support on $\R_+.$
Then we have after integration of~\eqref{eq:renewal} multiplied by $f$
\[\int_{\R_+} n(t,a)f(a)\,\d a=\int_{\R_+} n^{\rm in}(a)f(a)\,\d a+\int_0^t\int_{\R_+}n(s,a)\big(f'(a)-\beta(a) f(a)+\beta(a) f(0)\big)\,\d a\,\d s.\]
This motivates the definition of a measure solution to Equation~\eqref{eq:renewal}.
From now on we will denote by $\A$ the operator defined on $C^1_b(\R_+)=\{f\in C^1(\R_+)\,|\, f,f'\in C_b(\R_+)\}$ by
\[\A f(a):=f'(a)+\beta(a)(f(0)-f(a)).\]

\begin{definition}\label{def:meassol}
A family $(\mu_t)_{t\geq0}\subset\M(\R_+)$ is called a measure solution to Equation~\eqref{eq:renewal} with initial data $n^{\rm in}=\mu^{\rm in}\in\M(\R_+)$ if the mapping $t\mapsto\mu_t$ is narrowly continuous and
for all $f\in C^1_c(\R_+)$ and all $t\geq0,$
\beq\label{eq:meassol}\mu_tf=\mu^{\rm in} f+\int_0^t \mu_s\A f\,\d s,\eeq
{\it i.e.}
\[\int_{\R_+} f(a)\,\d\mu_t(a)=\int_{\R_+} f(a)\,\d\mu^{\rm in}(a)+\int_0^t\bigg(\int_{\R_+}\big[f'(a)-\beta(a) f(a)+\beta(a) f(0)\big]\,\d\mu_s(a)\bigg)\d s.\]
\end{definition}

\

\begin{proposition}\label{prop:sol_prop}
If $t\mapsto\mu_t$ is a solution to Equation~\eqref{eq:renewal} in the sense of the definition above, then for all $f\in C^1_b(\R_+)$ the function $t\mapsto\mu_tf$ is of class $C^1$ and satisfies
\beq\label{eq:meassol2}\left\{\begin{array}{l}
\frac{\d}{\d t}(\mu_tf)=\mu_t\A f,\qquad t\geq0
\vspace{2mm}\\
\mu_0=\mu^{\rm in}.
\end{array}\right.\eeq
Reciprocally any solution to~\eqref{eq:meassol2} satisfies~\eqref{eq:meassol}.
\end{proposition}

\begin{proof}
We start by checking that~\eqref{eq:meassol} is also satisfied for $f\in C^1_b(\R_+).$
Let $f\in C^1_b(\R_+),$ $t>0$ and $\rho\in C^1(\R)$ a nonincreasing function which satisfies
$\rho(a)=1$ for $a\leq0$ and $\rho(a)=0$ for $a\geq1.$
For all $n\in\N$ and $a\geq0$ define $f_n(a)=\rho(a-n)f(a).$
For all $n\in\N,$ $f_n\in C^1_c(\R_+)$ satisfies~\eqref{eq:meassol} and it remains to check that we can pass to the limit $n\to\infty.$
By monotone convergence we have $\mu_tf_n\to\mu_tf$ and $\mu^{\rm in}f_n\to\mu^{\rm in}f$ when $n\to\infty.$
Additionally $\|\A f_n\|_\infty\leq(\|\rho'\|_\infty+2\beta_{\max})\|f\|_\infty+\|f'\|_\infty$ 
so by dominated convergence we have $\mu_s\A f_n\to\mu_s\A f$ when $n\to\infty$ for all $s\in[0,t],$
and then $\int_0^t\mu_s\A f_n\,\d s\to\int_0^t\mu_s\A f\,\d s.$

For $f\in C^1_b(\R_+)$ we have $\A f\in C_b(\R_+)$ so $s\mapsto\mu_s\A f$ is continuous and using~\eqref{eq:meassol} we deduce that
\[\frac{\mu_{t+h} f-\mu_tf}{h}=\frac1h\int_{t}^{t+h}\mu_s\A f\,\d s\to\mu_t\A f\qquad\text{when}\ h\to0.\]
\end{proof}

We give now another equivalent notion of weak solutions to Equation~\eqref{eq:renewal}, which will be useful to prove uniqueness.
Compared to~\eqref{eq:meassol}, it uses test functions which depend on both variables $t$ and $a.$

\begin{proposition}\label{prop:weaksol}
A family $(\mu_t)_{t\geq0}\subset\M(\R_+)$
is a solution to Equation~\eqref{eq:renewal} in the sense of Definition~\ref{def:meassol}
if and only if the mapping $t\mapsto\mu_t$ is narrowly continuous and for all $\varphi\in C^1_c(\R_+\times\R_+)$
\beq\label{eq:weaksol}
\int_0^\infty\!\int_{\R_+}\big[\partial_t\varphi(t,a)+\partial_a\varphi(t,a)-\beta(a)\varphi(t,a)+\beta(a)\varphi(t,0)\big]\d\mu_t(a)\,\d t+\int_{\R_+}\varphi(0,a)\,\d\mu^{\rm in}(a)=0.
\eeq
This is also true by replacing $\varphi\in C^1_c(\R_+\times\R_+)$ by $\varphi\in C^1_b(\R_+\times\R_+)$ with compact support in time.
\end{proposition}

\begin{proof}
Assume that $(\mu_t)_{t\geq0}$ satisfies Definition~\ref{def:meassol} and let $\varphi\in C^1_b(\R_+\times\R_+)$ compactly supported in time.
As we have seen in the proof of Proposition~\ref{prop:sol_prop}, we can use $\partial_t\varphi(t,\cdot)\in C^1_b(\R_+)$ as a test function in~\eqref{eq:meassol}.
After integration in time we get
\begin{align*}
\int_0^\infty\!\int_{\R_+}\partial_t\varphi(t,a)\d\mu_t(a)\d t&=\int_0^\infty\!\int_{\R_+}\partial_t\varphi(t,a)\d\mu^{\rm in}(a)\d t\\
&\hspace{8mm}+\int_0^\infty\!\int_0^t\int_{\R_+}\big[\partial_a\partial_t\varphi(t,a)-\beta(a)\partial_t\varphi(t,a)+\beta(a)\partial_t\varphi(t,0)\big]\d\mu_s(a)\,\d s\,\d t\\
&=\int_{\R_+}\!\bigg(\int_0^\infty\partial_t\varphi(t,a)dt\bigg)\d\mu^{\rm in}(a)\\
&\hspace{8mm}+\int_0^\infty\!\int_{\R_+}\!\bigg(\int_s^\infty\big[\partial_t\partial_a\varphi(t,a)-\beta(a)\partial_t\varphi(t,a)+\beta(a)\partial_t\varphi(t,0)\big]\d t\bigg)\d\mu_s(a)\,\d s\\
&=-\int_{\R_+}\!\varphi(0,a)\d\mu^{\rm in}(a)-\int_0^\infty\!\int_{\R_+}\big[\partial_a\varphi(s,a)-\beta(a)\varphi(s,a)+\beta(a)\varphi(s,0)\big]\d\mu_s(a)\d s.
\end{align*}
Reciprocally let $T>0,$ $f\in C^1_c(\R_+),$ and assume that $t\mapsto\mu_t$ is narrowly continuous and satisfies~\eqref{eq:weaksol}.
We use the function $\rho$ defined in the proof of Proposition~\ref{prop:sol_prop} to define for $n\in\N$ and $t\geq0,$ $\rho_n(t)=\rho\big(n(t-T)\big).$
It is a decreasing sequence of decreasing $C^1_c(\R_+)$ functions which converges pointwise to $\1_{[0,T]}(t),$
and $\rho'_n$ (seen as an element of $\M(\R_+)$) converges narrowly to $-\delta_T.$
Using $\varphi_n(t,a)=\rho_n(t)f(a)$ as a test function in~\eqref{eq:weaksol} we get
\begin{align*}
0&=\int_0^\infty\!\int_{\R_+}\big[\rho_n'(t)f(a)+\rho_n(t)\A f(a)\big]\d\mu_t(a)\,\d t+\int_{\R_+}f(a)\,\d\mu^{\rm in}(a)\\
&=\int_0^\infty\rho_n'(t)\,(\mu_tf)\,\d t+\int_0^\infty\rho_n(t)\,(\mu_t\A f)\,\d t+\mu^{\rm in}f
\xrightarrow[n\to\infty]{}-\mu_Tf+\int_0^T\mu_t\A f\,\d t+\mu^{\rm in}f.
\end{align*}
\end{proof}

%

\section{The dual renewal equation}

To build a measure solution to Equation~\eqref{eq:renewal} we use a duality approach.
The idea is to start with the dual problem: find a solution to the dual renewal equation
\beq\label{eq:dual}
\p_tf(t,a)-\p_af(t,a)+\beta(a)f(t,a)=\beta(a)f(t,0),\qquad f(0,a)=f_0(a).
\eeq
As for the direct problem, we first give a definition of weak solutions for~\eqref{eq:dual} by using the method of characteristics.
Assume that $f\in C^1(\R_+\times\R_+)$ satisfies~\eqref{eq:dual} in the classical sense.
Then easy computations show that for all $a\geq0$ the function $\psi(t)=f(t,a-t)$ is solution to the ordinary differential equation $\psi'(t)+\beta(a-t)\psi(t)=0.$
After integration we get that $f$ satisfies
\[f(t,a-t)=f_0(a)\e^{-\int_0^t\beta(a-u)\d u}+\int_0^t\beta(a-\tau)\e^{-\int_\tau^t\beta(a-u)\d u}f(\tau,0)\,\d\tau\]
and the change of variable $a\leftarrow a+t$ leads to the following definition.

\

\begin{definition}
We say that $f\in C_b(\R_+\times\R_+)$ is a solution to~\eqref{eq:dual} when, for all $t,a\geq 0,$
\begin{align}
f(t,a)&=f_0(a+t)\e^{-\int_0^t\beta(a+u)\d u}+\int_0^t\e^{-\int_0^\tau\beta(a+u)\d u}\beta(a+\tau)f(t-\tau,0)\,\d\tau \label{eq:Duhamel1}\\
&=f_0(a+t)\e^{-\int_a^{a+t}\beta(u)\d u}+\int_a^{t+a}\e^{-\int_a^\tau\beta(u)\d u}\beta(\tau)f(a+t-\tau,0)\,\d\tau.\label{eq:Duhamel2}
\end{align}
\end{definition}

\noindent Formulations~\eqref{eq:Duhamel1} and~\eqref{eq:Duhamel2} are the same up to changes of variables, but both will be useful in the sequel.

\begin{theorem}\label{th:dual}
For all $f_0\in C_b(\R_+),$ there exists a unique $f\in C_b(\R_+\times\R_+)$ solution to~\eqref{eq:dual}.
Additionally
\[\begin{array}{l}
\text{for all}\ t\geq0,\ \|f(t,\cdot)\|_\infty\leq\|f_0\|_\infty,
\vspace{2mm}\\
\text{if}\ f_0\geq0\ \text{then for all}\ t\geq0,$ $f(t,\cdot)\geq0,
\vspace{2mm}\\
\text{if}\ f_0\in C^1_b(\R_+)\ \text{then for all}\ T>0,\ f\in C^1_b([0,T]\times\R_+).
\end{array}\]
\end{theorem}

\begin{proof}
To prove the existence and uniqueness of a solution we use the Banach fixed point theorem.
For $f_0\in C_b(\R_+)$ we define the operator $\Gamma:C_b([0,T]\times \R_+)\to C_b([0,T]\times\R_+)$ by
\[\Gamma f(t,a)= f_0(t+a)\e^{-\int_0^t\beta(a+u)\d u}+\int_0^t \e^{-\int_0^\tau\beta(a+u)\d u}\beta(a+\tau)f(t-\tau,0)\,\d\tau.\]
We easily have
\[\|\Gamma f-\Gamma g\|_\infty\leq T\|\beta\|_\infty\|f-g\|_\infty,\]
so $\Gamma$ is a contraction if $T<\|\beta\|_\infty$ and there is a unique fixed point in $C_b([0,T]\times\R_+).$
Additionally, since $\|f\|_\infty\leq\|f_0\|_\infty$ implies
\[\|\Gamma f\|_\infty\leq\|f_0\|_\infty\bigg(\e^{-\int_0^t\beta(a+u)\d u}+\int_0^t \e^{-\int_0^\tau\beta(a+u)\d u}\beta(a+\tau)\,\d\tau\bigg)=\|f_0\|_\infty,\]
the closed ball of radius $\|f_0\|_\infty$ is invariant under $\Gamma$ and the unique fixed point necessarily belongs to this ball.
Iterating on $[T,2T],$ $[2T,3T],$ $\dots,$ we obtain a unique global solution in $C_b(\R_+\times\R_+),$ and this solution satisfies $\|f\|_\infty\leq\|f_0\|_\infty.$
Since $\Gamma$ preserves non-negativity if $f_0\geq0,$
we get similarly that $f$ is nonnegative when $f_0$ is nonnegative.

\medskip

If $f_0\in C^1_b(\R_+)$ we can do the fixed point in $C^1_b([0,T]\times \R_+)$ endowed with the norm
\[\|f\|_{C^1_b}:=\|f\|_\infty+\|\p_tf\|_\infty+\|\p_af\|_\infty.\]
More precisely we do it in the closed and $\Gamma$-invariant subset $\{f\in C^1_b([0,T]\times \R_+),\ f(0,\cdot)=f_0\}.$
We have
\[\Gamma f(t,a)=f_0(t+a)\e^{-\int_a^{t+a}\beta(u)\d u}+\int_a^{t+a}\e^{-\int_a^\tau\beta(u)\d u}\beta(\tau)f(t-\tau+a,0)\,\d\tau\]
so by differentiation we get
\begin{align}
\p_t\Gamma f(t,a)&=[f_0'(t+a)-\beta(t+a)f_0(t+a)]\e^{-\int_a^{t+a}\beta(u)\d u}\nonumber\\
&\hspace{10mm}+\e^{-\int_a^{t+a}\beta(u)\d u}\beta(t+a)f_0(0)
+\int_a^{t+a}\e^{-\int_a^\tau\beta(u)\d u}\beta(\tau)\partial_tf(t-\tau+a,0)\,\d\tau\nonumber\\
&=\A f_0(t+a)\e^{-\int_a^{t+a}\beta(u)\d u}+\int_a^{t+a}\e^{-\int_a^\tau\beta(u)\d u}\beta(\tau)\partial_tf(t-\tau+a,0)\,\d\tau.\label{eq:d_tGamma}
\end{align}
and
\begin{align*}
\p_a\Gamma f(t,a)&=\big[f_0'(t+a)+(\beta(a)-\beta(t+a))f_0(t+a)\big]\e^{-\int_a^{t+a}\beta(u)\d u}
+\e^{-\int_a^{t+a}\beta(u)\d u}\beta(t+a)f_0(0)-\beta(a)f(t,0)\\
&\hspace{10mm}
+\int_a^{t+a}\beta(a)\e^{-\int_a^\tau\beta(u)\d u}\beta(\tau)f(t-\tau+a,0)\,\d\tau
+\int_a^{t+a}\e^{-\int_a^\tau\beta(u)\d u}\beta(\tau)\partial_tf(t-\tau+a,0)\,\d\tau\\
&=\big[f_0'(t+a)+(\beta(a)-\beta(t+a))(f_0(t+a)-f_0(0))\big]\e^{-\int_a^{t+a}\beta(u)\d u}\\
&\hspace{60mm}
+\int_a^{t+a}\e^{-\int_a^\tau\beta(u)\d u}(\beta(\tau)-\beta(a))\partial_tf(t-\tau+a,0)\,\d\tau.
\end{align*}
Finally we get
\[\|\Gamma f-\Gamma g\|_{C^1_b}\leq T\|\beta\|_\infty\big[\|f-g\|_\infty+2\|\p_t(f-g)\|_\infty\big]
\leq 2T\|\beta\|_\infty\|f-g\|_{C^1_b}.\]
We conclude that if $f_0\in C^1_b(\R_+)$ then the unique solution belongs to $C^1_b([0,T]\times \R_+)$ for all $T>0.$

\end{proof}

\begin{lemma}\label{lm:finite_prop}
Let $f_0,g_0\in C_b(\R_+)$ such that
\[\exists A>0,\ \forall a\in[0,A],\ f_0(a)=g_0(a)\]
and let $f$ and $g$ the solutions to Equation~\eqref{eq:dual} with initial distributions $f_0$ and $g_0$ respectively.
Then for all $T\in(0,A)$ we have
\[\forall t\in[0,T],\ \forall a\in[0,A-T],\qquad f(t,a)=g(t,a).\]
\end{lemma}

\begin{proof}
The closed subspace $\{f\in C_b([0,T]\times \R_+),\ \forall t\in[0,T],\,\forall a\in[0,A-T],\ f(t,a)=g(t,a)\}$ is invariant under $\Gamma$ if $f_0$ satisfies $\forall a\in[0,A],\ f_0(a)=g_0(a).$
\end{proof}

\begin{proposition}\label{prop:rightMt}
The family of operators $(M_t)_{t\geq0}$ defined on $C_b(\R_+)$ by
$M_tf_0=f(t,\cdot)$
is a semigroup, {\it i.e.}
\[M_0f=f\qquad and\qquad M_{t+s}f=M_t(M_sf).\]
Additionally it is a positive and conservative contraction, {\it i.e} for all $t\geq0$ we have
\begin{align*}
&f\geq0\ \implies\ M_tf\geq0,\\
&M_t\mathbf1=\mathbf1,\\
&\forall f\in C_b(\R_+),\quad\|M_tf\|_\infty\leq\|f\|_\infty.
\end{align*}
Finally it satisfies, for all $f\in C^1_b(\R_+)$ and all $t>0,$
\[\partial_t M_tf=M_t\A f=\A M_tf.\]
\end{proposition}

\begin{proof}
Let $f\in C_b(\R_+),$ fix $s\geq0,$ and define $g(t,a)=M_t(M_sf)(a)$ and $h(t,a)=M_{t+s}f(a).$
We have $g(0,a)=h(0,a)$ and
\[g(t,a)=M_sf(t+a)\e^{-\int_0^t\beta(a+u)\d u}
+\int_0^t\e^{-\int_0^\tau\beta(a+u)\d u}\beta(a+\tau)g(t-\tau,0)\,\d\tau\]
and
\begin{align*}
h(t,a)&=f(t+s+a)\e^{-\int_0^{t+s}\beta(a+u)du}
+\int_0^t\e^{-\int_0^\tau\beta(a+u)\d u}\beta(a+\tau)h(t-\tau,0)\,\d\tau\\
&\hskip12mm +\int_t^{t+s}\e^{-\int_0^\tau\beta(a+u)\d u}\beta(a+\tau)M_{t+s-\tau}f(0)\,\d\tau\\
&=f(t+s+a)\e^{-\int_0^{t+s}\beta(a+u)\d u}
+\int_0^t\e^{-\int_0^\tau\beta(a+u)\d u}\beta(a+\tau)h(t-\tau,0)\,\d\tau\\
&\hskip12mm +\int_0^s\e^{-\int_0^{\tau+t}\beta(a+u)\d u}\beta(a+\tau+t)M_{s-\tau}f(0)\,\d\tau\\
&=\bigg[f(t+a+s)\e^{-\int_0^s\beta(t+a+u)\d u}
+\int_0^s\e^{-\int_0^\tau\beta(t+a+u)\d u}\beta(t+a+\tau)M_{s-\tau}f(0)\,d\tau\bigg]\e^{-\int_0^t\beta(a+u)\d u}\\
&\hskip12mm +\int_0^t\e^{-\int_0^\tau\beta(a+u)\d u}\beta(a+\tau)h(t-\tau,0)\,\d\tau\\
&=M_sf(t+a)\e^{-\int_0^t\beta(a+u)\d u}+\int_0^t\e^{-\int_0^\tau\beta(a+u)\d u}\beta(a+\tau)h(t-\tau,0)\,\d\tau.
\end{align*}
By uniqueness of the fixed point we deduce that $g(t,a)=h(t,a)$ for all $t\geq0.$

The conservativeness $M_t\mathbf1=\mathbf1$ is straightforward computations and the positivity and the contraction property follow immediately from Theorem~\ref{th:dual}.

Now consider $f\in C^1_b(\R_+).$
From~\eqref{eq:d_tGamma} we deduce that $\partial_tM_tf$ is the fixed point of $\Gamma$ with initial data $\A f.$
By uniqueness of the fixed point we deduce that $\partial_tM_tf=M_t\A f.$
But since $f\in C^1_b(\R_+),$ $M_tf$ satisfies~\eqref{eq:dual} in the classical sense, {\it i.e.} $\partial_t M_tf(a)=\A M_tf(a),$ and the proof is complete.
\end{proof}

\begin{remark}
The semigroup $(M_t)_{t\geq0}$ is not strongly continuous, {\it i.e.} we do not have $\lim_{t\to0}\|M_tf-f\|_\infty=0$ for all $f\in C_b(\R_+).$
However the restriction of the semigroup to the invariant subspace of bounded and uniformly continuous functions is strongly continuous.
\end{remark}

\section{Existence and uniqueness of a measure solution}

We are now ready to prove the existence and uniqueness of a measure solution to Equation~\eqref{eq:renewal}.
We define the dual semigroup $(M_t)_{t\geq0}$ on $\M(\R_+)=C_0(\R_+)'$ by
\[\mu M_t:f\in C_0(\R_+)\mapsto\mu(M_tf)=\int_{\R_+}M_tf\,\d\mu.\]
In other words we have by definition
\[\forall f\in C_0(\R_+),\qquad(\mu M_t)f=\mu(M_tf),\qquad{\it i.e.}\quad \int_{\R_+}f\,\d(\mu M_t)=\int_{\R_+}M_tf\,\d\mu.\]
The following lemma ensures that this identity is also satsified for $f\in C_b(\R_+).$
From now on we will denote without ambiguity the quantity $(\mu M_t)f=\mu(M_tf)$ by $\mu M_tf.$

\begin{lemma}
For all $f\in C_b(\R_+)$ we have $(\mu M_t)f=\mu(M_tf).$
\end{lemma}

\begin{proof}
The identity is true by definition for any $f\in C_0(\R_+).$
Let $f\in C_b(\R_+)$ and $f_n$ as defined in~\eqref{eq:f_n}.
The sequence $(f_n)_{n\in\N}$ lies in $C_0(\R_+)$ so $(\mu M_t)f_n=\mu(M_tf_n)$ for all $n\in\N.$
By monotone convergence we clearly have $(\mu M_t)f_n\to(\mu M_t)f.$
We will prove that $\mu (M_tf_n)\to\mu (M_tf)$ by dominated convergence.
First by positivity of $M_t$ we have $|M_tf_n(a)|\leq M_t|f_n|(a)\leq M_t|f|(a)\leq\|f\|_\infty.$ 
Additionally since $f_n(a)=f(a)$ for all $a\in[0,n],$ we have from Lemma~\ref{lm:finite_prop} that for all $a\in[0,n-t],$ $M_tf_n(a)=M_tf(a),$
so for all $a\geq0$ we have $M_tf_n(a)\to M_tf(a).$
\end{proof}

\begin{proposition}
The left semigroup $(M_t)_{t\geq0}$ is a positive and conservative contraction, {\it i.e} for all $t\geq0$ we have
\[\mu\geq0\quad\implies\quad\mu M_t\geq0\quad \text{and}\quad \|\mu M_t\|_{TV}=\|\mu\|_{TV},\]
\[\forall\mu\in\M(\R_+),\qquad\|\mu M_t\|_{TV}\leq\|\mu\|_{TV}.\]
\end{proposition}

\begin{proof}
It is a consequence of Proposition~\ref{prop:rightMt}.
If $\mu\geq0,$ then for any $f\geq0$ we have $(\mu M_t) f=\mu(M_t f)\geq0.$
Additionnally $\|\mu M_t\|_{TV}=(\mu M_t)\mathbf1=\mu(M_t\mathbf1)=\mu\mathbf1=\|\mu\|_{TV}.$
For $\mu\in\M(\R_+)$ not necessarily positive, we have
\[\|\mu M_t\|_{TV}=\sup_{\|f\|_\infty\leq1}|\mu M_tf|\leq\sup_{\|g\|_\infty\leq1}|\mu g|=\|\mu\|_{TV}.\]
\end{proof}

\begin{remark}
The left semigroup is not strongly continuous, {\it i.e.} we do not have $\lim_{t\to0}\|\mu M_t-\mu\|_{TV}=0$ for all $\mu\in\M(\R_+).$
This is due to the non continuity of the transport semigroup for the total variation distance: for instance for $a\in\R$ we have $\|\delta_{a+t}-\delta_a\|_{TV}=2$ for any $t>0.$
But the left semigroup is weak* continuous.
This is an immediate consequence of the strong continuity of the right semigroup on the space of bounded and uniformly continuous functions (which contains $C_0(\R_+)$).
The left semigroup is even narrowly continuous as we will see in the proof of the next theorem.
\end{remark}

\begin{theorem}
For any $\mu^{\rm in}\in\M(\R_+),$ the orbit map
$t\mapsto\mu^{\rm in}M_t$
is the unique measure solution to Equation~\eqref{eq:renewal}.
\end{theorem}

\begin{proof}\emph{Existence.}
In this part we use Proposition~\ref{prop:rightMt} at different places.
Let $\mu\in\M(\R_+).$
We start by checking that $t\mapsto\mu M_t$ is narrowly continuous.
Let $f\in C_b(\R_+).$
Due to the semigroup property, it is sufficient to check that $\lim_{t\to0}\mu M_tf=\mu f.$
But from~\eqref{eq:Duhamel1} we have
\[|\mu M_tf-\mu f|\leq\underbrace{\bigg|\mu f-\int_{\R_+}f(a+t)\e^{-\int_0^t\beta(a+u)\d u}\d\mu(a)\bigg|}_{\xrightarrow[t\to0]{}\,0\ \text{by dominated convergence}}
+\underbrace{\bigg|\mu\int_0^t\e^{-\int_0^\tau\beta(\cdot+u)\d u}\beta(\cdot+\tau)M_{t-\tau}f(0)\,\d\tau\bigg|}_{\leq\|\mu\|_{TV}\|\beta\|_\infty\|f\|_\infty t\,\xrightarrow[t\to0]{}\,0}.\]

Now we check that $(\mu M_t)_{t\geq0}$ satisfies~\eqref{eq:meassol}.
Let $f\in C^1_c(\R_+).$
Using the identity $\partial_tM_tf=M_t\A f$ and the Fubini's theorem we have
\[(\mu M_t) f=\mu(M_tf)=\mu f+\mu\int_0^t\p_s M_sf\,\d s=\mu f+\mu\int_0^t M_s\A f\,\d s=\mu f+\int_0^t(\mu M_s)\A f\,\d s.\]

\medskip

\emph{Uniqueness.}
Because of Proposition~\eqref{prop:weaksol} we can prove uniqueness on formulation~\eqref{eq:weaksol}.
By linearity we can assume that $\mu^{\rm in}=0$
and we want to prove that the unique family $(\mu_t)_{t\geq0}$ which satisfies
\[\int_0^\infty\!\int_{\R_+}\big[\partial_t\varphi(t,a)+\partial_a\varphi(t,a)-\beta(a)\varphi(t,a)+\beta(a)\varphi(t,0)\big]\d\mu_t(a)\,\d t=0\]
for all $\varphi\in C^1_b(\R_+\times\R_+)$ with compact support in time, is the trivial family.
If we can prove that for all $\psi\in C^1_c(\R_+\times\R_+)$ there exists $\varphi\in C^1_b(\R_+\times\R_+)$ compactly supported in time such that
for all $t,a\geq0$
\beq\label{eq:dual_inh}
\partial_t\varphi(t,a)+\partial_a\varphi(t,a)-\beta(a)\varphi(t,a)+\beta(a)\varphi(t,0)=\psi(t,a),
\eeq
then we get the conclusion.
Let $\psi\in C^1_c(\R_+\times\R_+)$ and let $T>0$ such that $\supp\psi\subset[0,T)\times\R_+.$
Using the same method as for~\eqref{eq:dual}, we can prove the existence of a solution $\varphi\in C^1_b([0,T]\times\R_+)$ to~\eqref{eq:dual_inh}
with terminal condition $\varphi(T,a)=0.$
Since $\psi\in C^1_c([0,T)\times\R_+)$ we easily check that the extension of $\varphi(t,a)$ by $0$ for $t>T$ belongs to $C^1_b(\R_+\times\R_+),$ is compactly supported in time, and satisfies~\eqref{eq:dual_inh}.

\end{proof}

\section{Exponential convergence to the invariant measure}
\label{sec:asymptotic}

As noticed first by Sharpe and Lotka in~\cite{SharpeLotka}
the asymptotic behavior of the renewal equation consists in a convergence to a stationary distribution.
This property has then been proved by many authors using various methods for $L^1$ solutions
~\cite{Feller,Greiner,KhaladiArino,GwiazdaPerthame,PakdamanPerthameSalort10,PakdamanPerthameSalort13,Webb84}.
The generalized entropy method has even been extended to measure solutions in~\cite{GwiazdaWiedemann}.
Here we use a different approach which is based on a coupling argument.
More precisely we prove a so-called Doeblin's condition (see~\cite{MeynTweedie} for instance) which guarantees exponential convergence of any measure solution to the stationnary distribution.
Denote by $\P(\R_+)$ the set of probability measures, namely the positive measures with mass 1.

\begin{theorem}
Let $(M_t)_{t\geq0}$ be a conservative semigroup on $\M(\R_+)$ which satisfies the Doeblin's condition
\[\exists\, t_0>0,\, c\in(0,1),\, \nu\in\P(\R_+)\quad\text{such that}\quad \forall \mu\in\P(\R_+),\quad \mu M_{t_0}\geq c\,\nu.\]
Then for $\alpha:=\frac{-\log(1-c)}{t_0}>0$ we have for all $\mu_1,\mu_2\in\M(\R_+)$ such that $\mu_1(\R_+)=\mu_2(\R_+)$ and all $t\geq0$
\[\|\mu_1M_t-\mu_2M_t\|_{TV}\leq \e^{-\alpha(t-t_0)}\|\mu_1-\mu_2\|_{TV}.\]
\end{theorem}

\begin{proof}
Let $\mu_1,\mu_2\in\M(\R_+)$ such that $\mu_1(\R_+)=\mu_2(\R_+),$
so that $(\mu_1-\mu_2)_+(\R_+)=(\mu_2-\mu_1)_+(\R_+)=\frac12\|\mu_1-\mu_2\|_{TV}.$
The measure $\mu$ defined by
\[\mu:=\frac{2}{\|\mu_1-\mu_2\|_{TV}}(\mu_1-\mu_2).\]
thus satifies $\mu_+,\mu_-\in\P(\R_+)$ 
and the Doeblin's condition ensures that
\[\mu_\pm M_{t_0}\geq c\,\nu.\]
Using the conservativeness of $M_{t_0}$ we deduce that
\[\|\mu_\pm M_{t_0}-c\,\nu\|_{TV}=(\mu_\pm M_{t_0}-c\,\nu)(\R_+)=1-c\]
and then
\[\|\mu M_{t_0}\|_{TV}\leq\|\mu_+ M_{t_0}-c\,\nu\|_{TV}+\|\mu_- M_{t_0}-c\,\nu\|_{TV}=2(1-c)\]
which gives
\[\|\mu_1M_{t_0}-\mu_2M_{t_0}\|_{TV}=\frac12\|\mu_1-\mu_2\|_{TV}\|\mu M_{t_0}\|_{TV}\leq(1-c)\|\mu_1-\mu_2\|_{TV}.\]
Now for $t\geq0$ we define $n=\big\lfloor\frac{t}{t_0}\big\rfloor$ and we get by induction
\[\|\mu_1 M_t-\mu_2 M_t\|_{TV}\leq(1-c)^n\|\mu_1 M_{t-nt_0}-\mu_2 M_{t-nt_0}\|_{TV}\leq \e^{n\log(1-c)}\|\mu_1-\mu_2\|_{TV}.\]
This ends the proof since
\[n\log(1-c)\leq\Big(\frac{t}{t_0}-1\Big)\log(1-c)=-\alpha(t-t_0).\]

\end{proof}

\begin{proposition}
The renewal semigroup $(M_t)_{t\geq0}$ satisfies the Doeblin's condition
with $t_0=a_*+\eta,$ $c=\eta\,\beta_{\min}\,\e^{-\beta_{\max}(a_*+\eta)},$
and $\nu$ the uniform probability measure on $[0,\eta],$ for any choice of $\eta>0.$
\end{proposition}

\begin{proof}
We prove the equivalent dual formulation of the Doeblin's condition:
\[\exists\, t_0>0,\, c\in(0,1),\, \nu\in\P(\R_+)\quad\text{such that}\quad \forall f\geq0,\ \forall a\geq0,\quad M_{t_0}f(a)\geq c\,(\nu f).\]
To do so we iterate once the Duhamel formula~\eqref{eq:Duhamel1} and we get for any $f\geq0$
\begin{align*}
f(t,a)&=f_0(t+a)\e^{-\int_0^t\beta(a+u)\d u}+\int_0^t\e^{-\int_0^\tau\beta(a+u)\d u}\beta(a+\tau)f(t-\tau,0)\,\d\tau\\
&=f_0(t+a)\e^{-\int_0^t\beta(a+u)\d u} +\int_0^t\e^{-\int_0^\tau\beta(a+u)\d u}\beta(a+\tau)f_0(t-\tau)\e^{-\int_0^{t-\tau}\beta(u)\d u}\,\d\tau
+(\geq0)\\
&\geq\int_0^t\e^{-\int_0^\tau\beta(a+u)\d u}\beta(a+\tau)f_0(t-\tau)\e^{-\int_0^{t-\tau}\beta(u)\d u}\,\d\tau\\
&=\int_0^t\e^{-\int_0^{t-\tau}\beta(a+u)\d u}\beta(a+t-\tau)f_0(\tau)\e^{-\int_0^{\tau}\beta(u)\d u}\,\d\tau.
\end{align*}
Consider the probability measure $\nu$ defined by $\nu f=\frac1\eta\int_0^\eta f(a)\,\d a$ for some $\eta>0.$
Then for any $t\geq a_*+\eta$ and any $a\geq0$ we have
\begin{align*}
f(t,a)&\geq\int_0^t\e^{-\int_0^{t-\tau}\beta(a+u)\d u}\beta(a+t-\tau)f_0(\tau)e^{-\int_0^{\tau}\beta(u)\d u}\,\d\tau\\
&\geq\int_0^\eta \e^{-\beta_{\max}(t-\tau)}\beta_{\min}f_0(\tau)\e^{-\beta_{\max}\tau}\,\d\tau\\
&=\eta\,\beta_{\min}\,\e^{-\beta_{\max}t}(\nu f_0).
\end{align*}
\end{proof}

As we have already seen in the introduction, the conservative equation~\eqref{eq:renewal} admits a unique invariant probability measure,
{\it i.e.} there exists a unique probability measure $\mu_\infty$ such that $\mu_\infty M_t=\mu_\infty$ for all $t\geq0.$
This probability measure has a density with respect to the Lebesgue measure
\[\d\mu_\infty=N(a)\,\d a\]
where $N$ is explicitly given by
\[N(a)=N(0)\,\e^{-\int_0^a\beta(u)\d u}\]
with $N(0)$ such that $\int_0^\infty N(a)\,\d a=1.$

\begin{corollary}
For all $\mu\in\M(\R_+)$ and all $\eta>0$ we have
\[\forall t\geq0,\qquad\|\mu M_t-(\mu\mathbf1)\mu_\infty\|_{TV}\leq \e^{-\alpha(t-t_0)}\|\mu-(\mu\mathbf1)\mu_\infty\|_{TV},\]
where $t_0=a_*+\eta,$ $c=\eta\,\beta_{\min}\,\e^{-\beta_{\max}(a_*+\eta)},$ and $\alpha=\frac{-\log(1-c)}{t_0}.$
\end{corollary}

Notice that in the case $a_*=0$ we obtain by passing to the limit $\eta\to0$
\[\forall t\geq0,\qquad\|\mu M_t-(\mu\mathbf1)\mu_\infty\|_{TV}\leq \e^{-\beta_{\min}t}\,\|\mu-(\mu\mathbf1)\mu_\infty\|_{TV}.\]

\section*{Acknowledgments}
The author is very thankful to Vincent Bansaye and Bertrand Cloez for having initiated him to the Doeblin's method.
This work has been partially supported by the ANR project KIBORD, ANR-13-BS01-0004, funded by the French Ministry of Research.



\begin{thebibliography}{10}

\bibitem{CanizoCarrilloCuadrado}
J.~A. Ca\~nizo, J.~A. Carrillo, and S.~Cuadrado.
\newblock Measure solutions for some models in population dynamics.
\newblock {\em Acta Appl. Math.}, 123:141--156, 2013.

\bibitem{CarrilloColomboGwiazdaUlikowska}
J.~A. Carrillo, R.~M. Colombo, P.~Gwiazda, and A.~Ulikowska.
\newblock Structured populations, cell growth and measure valued balance laws.
\newblock {\em J. Differential Equations}, 252(4):3245--3277, 2012.

\bibitem{DiekmannGetto}
O.~Diekmann and P.~Getto.
\newblock Boundedness, global existence and continuous dependence for nonlinear
  dynamical systems describing physiologically structured populations.
\newblock {\em J. Differential Equations}, 215(2):268--319, 2005.

\bibitem{EversHilleMuntean15}
J.~H.~M. Evers, S.~C. Hille, and A.~Muntean.
\newblock Mild solutions to a measure-valued mass evolution problem with flux
  boundary conditions.
\newblock {\em J. Differential Equations}, 259(3):1068--1097, 2015.

\bibitem{Feller}
W.~Feller.
\newblock On the integral equation of renewal theory.
\newblock {\em Ann. Math. Statistics}, 12:243--267, 1941.

\bibitem{Greiner}
G.~Greiner.
\newblock A typical {P}erron-{F}robenius theorem with applications to an
  age-dependent population equation.
\newblock In {\em Infinite-dimensional systems ({R}etzhof, 1983)}, volume 1076
  of {\em Lecture Notes in Math.}, pages 86--100. Springer, Berlin, 1984.

\bibitem{GwiazdaLorenzMarciniak}
P.~Gwiazda, T.~Lorenz, and A.~Marciniak-Czochra.
\newblock A nonlinear structured population model: {L}ipschitz continuity of
  measure-valued solutions with respect to model ingredients.
\newblock {\em J. Differential Equations}, 248(11):2703--2735, 2010.

\bibitem{GwiazdaPerthame}
P.~Gwiazda and B.~Perthame.
\newblock Invariants and exponential rate of convergence to steady state in the
  renewal equation.
\newblock {\em Markov Process. Related Fields}, 12(2):413--424, 2006.

\bibitem{GwiazdaWiedemann}
P.~Gwiazda and E.~Wiedemann.
\newblock Generalized entropy method for the renewal equation with measure
  data.
\newblock {\em Commun. Math. Sci.}, 15(2):577--586, 2017.

\bibitem{Iannelli}
M.~Iannelli.
\newblock {\em Mathematical Theory of Age-Structured Population Dynamics}.
\newblock Applied Math. Monographs, CNR, Giardini Editori e Stampatori, Pisa,
  1995.

\bibitem{KhaladiArino}
M.~Khaladi and O.~Arino.
\newblock Estimation of the rate of convergence of semigroups to an
  asynchronous equilibrium.
\newblock {\em Semigroup Forum}, 61(2):209--223, 2000.

\bibitem{MetzDiekmann}
J.~A.~J. Metz and O.~Diekmann, editors.
\newblock {\em The dynamics of physiologically structured populations},
  volume~68 of {\em Lecture Notes in Biomathematics}.
\newblock Springer-Verlag, Berlin, 1986.
\newblock 

\bibitem{MeynTweedie}
S.~P. Meyn and R.~L. Tweedie.
\newblock {\em Markov chains and stochastic stability}.
\newblock Communications and Control Engineering Series. Springer-Verlag
  London, Ltd., London, 1993.

\bibitem{PakdamanPerthameSalort10}
K.~Pakdaman, B.~Perthame, and D.~Salort.
\newblock Dynamics of a structured neuron population.
\newblock {\em Nonlinearity}, 23(1):55--75, 2010.

\bibitem{PakdamanPerthameSalort13}
K.~Pakdaman, B.~Perthame, and D.~Salort.
\newblock Relaxation and self-sustained oscillations in the time elapsed neuron
  network model.
\newblock {\em SIAM J. Appl. Math.}, 73(3):1260--1279, 2013.

\bibitem{Perthame}
B.~Perthame.
\newblock {\em Transport equations in biology}.
\newblock Frontiers in Mathematics. Birkh\"auser Verlag, Basel, 2007.

\bibitem{Rudin}
W.~Rudin.
\newblock {\em Real and complex analysis}.
\newblock McGraw-Hill Book Co., New York, third edition, 1987.

\bibitem{SharpeLotka}
F.~R. Sharpe and A.~J. Lotka.
\newblock {\em A Problem in Age-Distribution}, pages 97--100.
\newblock Springer Berlin Heidelberg, Berlin, Heidelberg, 1977.

\bibitem{Thieme}
H.~R. Thieme.
\newblock {\em Mathematics in population biology}.
\newblock Princeton Series in Theoretical and Computational Biology. Princeton
  University Press, Princeton, NJ, 2003.

\bibitem{Webb84}
G.~F. Webb.
\newblock A semigroup proof of the {S}harpe-{L}otka theorem.
\newblock In {\em Infinite-dimensional systems ({R}etzhof, 1983)}, volume 1076
  of {\em Lecture Notes in Math.}, pages 254--268. Springer, Berlin, 1984.

\bibitem{Webb}
G.~F. Webb.
\newblock {\em Theory of nonlinear age-dependent population dynamics},
  volume~89 of {\em Monographs and Textbooks in Pure and Applied Mathematics}.
\newblock Marcel Dekker, Inc., New York, 1985.

\end{thebibliography}
\end{document}